\documentclass{amsart} 
\usepackage{amsmath,amscd} 
\usepackage{latexsym,amsmath,amssymb,mathrsfs,setspace,enumerate}
\usepackage[T1]{fontenc}
\usepackage[utf8]{inputenc}
\usepackage{lmodern} 
\usepackage{amssymb}
\usepackage{tikz-cd}
\usepackage{setspace}
\usepackage{graphics}
\usepackage{graphicx}
\usepackage[utf8]{inputenc}
\usepackage[english]{babel}
\usepackage{fancyhdr}
\usepackage{tikz}
\usetikzlibrary{decorations.markings,positioning}
\theoremstyle{plain}

\newtheorem{lemma}{Lemma} 
 
\newtheorem{thm}{Theorem}
\newtheorem{exam}{Example}

\newtheorem{defn}{Definition}

\newtheorem{rmk}[thm]{Remark}

\tikzcdset{%
	triple line/.code={\tikzset{%
			double distance = 3pt, 
			double=\pgfkeysvalueof{/tikz/commutative diagrams/background color}}},
	quadruple line/.code={\tikzset{%
			double distance = 5.3pt, 
			double=\pgfkeysvalueof{/tikz/commutative diagrams/background color}}},
	Rrightarrow/.code={\tikzcdset{triple line}\pgfsetarrows{tikzcd implies cap-tikzcd implies}},
	RRightarrow/.code={\tikzcdset{quadruple line}\pgfsetarrows{tikzcd implies cap-tikzcd implies}}
}

\title{Representations of Plesken Lie algebras}
\author {P. G. Romeo and S. N. Arjun}
\address{Dept. of Mathematics, Cochin University of Science and Technology, Kochi, Kerala, INDIA.}
\email{$romeo_-parackal@yahoo.com,\, arjunsnmaths1996@gmail.com$}
\thanks{} 
\date{}

\begin{document}
	
	\begin{abstract}
		Cohen and Taylor introduced Plesken Lie algebras of finite groups and studied their  structural properties. As a further step, we will introduce Plesken Lie algebra representations, Plesken Lie algebra modules and discuss its irreducibility.
	\end{abstract}
	\maketitle
	

Representations of Lie groups and Lie algebras is an active area of research related to 
 Lie theory which has its roots in the work of Sophus Lie, who studied certain transformation groups that are now called Lie groups. At present both Lie groups and Lie algebras have become essential to many parts of mathematics and theoretical physics. 
 The Lie algebra, being a linear object, is more immediately accessible than the Liegroup. It was Wilhelm Killing (1847–1923) who insisted that before one could classify all group actions one should begin by classifying all (finite dimensional real) Lie algebras. The gradual evolution of the ideas of Lie, Friedrich Engel (1861–1941), and Killing, made it clear that determining all simple Lie algebras was fundamental. Throughout this paper, unless otherwise stated, a Lie algebra is a finite dimensional over a field of characteristic zero. \\
 
A group algebra of a finite group $G$ is a vector space having basis consisting of group elements which also carries extra structure involving the product operation on $G$. In a sense, group algebras are the source of all you need to know about representation theory of finite groups.  A representation of a finite group $G$ on a finite dimensional complex vector space $V$ is a homomorphism $\rho : G \to GL(V)$ of $G$ to the group of automorphisms of $V$. A subrepresentation of a representation $V$ is a vector subspace $W$ of $V$ which is invariant under $G$. That is $\rho(g)(w) \in W$ for all $g \in G$ and $w \in W$. A representation $\rho$ is \textit{irreducible} if there is no proper invariant subspace $W$ of $V$. The group algebra assumes the structure of a Lie algebra via the bracket (commutator) operation.\\

W. Plseken and Arjeh M. Cohen described the construction of a Lie algebra of a finite group and call it as Plesken Lie algebra. The question that they encountered was for which groups the construction would produce a simple Lie algebra? In \cite{Plesken}, Arjeh M. Cohen and D. E. Taylor studied the structure of Plesken Lie algebras in detail by explicitely determining the groups for which the Plesken Lie algebra is simple or semisimple. In this paper we discuss the representations of Plesken Lie algebras,  Plesken Lie algebra modules and the Plesken Lie algebra representations induced from the group representations. Further we also  determine the irreducible Plesken Lie algebra representations.

\section{Preliminaries}	
	We will first go through preliminaries of Lie algebras, group algebras and Plesken Lie algebras. Also we describes some structural properties of group algebra. Further we will introduce the notion of representation of Plesken Lie algebras. We also discuss some irreducible representations of Plesken Lie algebras of some groups.\\
	
	A Lie algebra $L$ over an arbitrary field $\mathbb{F}$  is a vector space over $\mathbb{F}$ endowed with an operation called Lie bracket satisfying the following properties:
	\begin{enumerate}
		\item Bilinearity :  For $x,y,z \in L, a,b \in \mathbb{F}$
		\begin{center}
			$[ax+by, z]= a[x, z]+b[y, z]$\\
			$[x, ay+bz] = a[x,y]+b[x,z]$
		\end{center}
		\item $[x,x] =0$ for all $x \in L$
		\item Jacobi identity :\\$[x,[y,z]]+[y,[z,x]]+[z,[x,y]]=0$ for all $x,y,z \in L$  
	\end{enumerate}
	A vector subspace $K \subset L$ is a \textit{Lie subalgebra} of $L$ if $[x, y] \in K$ for all $x, y \in K$, and a subspace $I$ of a Lie algebra $L$ is called an \textit{ideal} if $[x, y] \in I$ for all $x \in L, y \in I$. A \textit{Lie algebra homomorphism} from $L$ to $L'$ is a linear map $\phi : L \to L'$ such that $\phi([x, y]) = [\phi(x), \phi(y)]$ for all $x, y \in L$.

	A group algebra $\mathbb{F}G$ of a group $G$ over $\mathbb{F}$ is a vector space with group elemensts as basis. That is,
	\begin{center}
		$\mathbb{F}G=\{ \displaystyle\sum_{i=1}^n a_ig_i: a_i \in \mathbb{F} \ \text{for all i} \}$
	\end{center}
	Here the multiplication is defined by 
	\begin{center}
		$\begin{pmatrix} \displaystyle\sum_{i=1}^n a_ig_i \end{pmatrix}
		\begin{pmatrix} \displaystyle\sum_{j=1}^nb_jg_j \end{pmatrix}=\displaystyle\sum_{i,j =1}^na_i b_j(g_ig_j) , \ \indent a_i,b_j \in \mathbb{F} $
	\end{center} 
	
	A linear map between two group algebras ${\mathbb{F}G}$ and ${\mathbb{F}H}$ which preserves the algebra multiplication  is a \textit{homomorphism} of group  algebras.
	
	\begin{defn}(cf.\cite{Plesken}) Plesken Lie algebra $\mathcal{L}(G)$ of a  group $G$ over $\mathbb{F}$   is the linear span of elements $\hat{g} = g-{g}^{-1} \in \mathbb{F}G$ together with the Lie bracket
		\begin{center}
			$[\hat{g}, \hat{h}] = \hat{g}\hat{h} - \hat{h}\hat{g}$
		\end{center}
	\end{defn}	
	\begin{exam} 
		Consider the symmetric group $S_3$, then  
		\begin{equation*}
			\begin{split}
				L(S_3) &= \text{span} \{ \sigma - {\sigma}^{-1} : \sigma \in S_3 \}\\
				&= \{ a_1((1) - (1)) + a_2((1 \ 2) - (1 \ 2)) +a_3((1 \ 3) - (1 \ 3)) + +a_4((2 \ 3) - (2 \ 3)) \\
				& \ \indent \quad + a_5((1 \ 2 \ 3) - (1 \ 3 \ 2))+ a_6((1 \ 3 \ 2) - (1 \ 2 \ 3)) : a_i \in \mathbb{C} \}\\
				&= \{ a((1 \ 2 \ 3) - (1 \ 3 \ 2)) : a \in \mathbb{C} \}
			\end{split}
		\end{equation*} 
		is a one dimensional  Plesken Lie algebra over $\mathbb{C}$ with Lie bracket
		\begin{center}
			$[a\widehat{(1 \ 2 \ 3)}, b \widehat{(1 \ 2 \ 3)}] = 0$
		\end{center}
	\end{exam}

\section{Group algebra homomorphisms and Plesken Lie algebra homomorphisms}
Here we go through some structural properties of group algebras and Plesken Lie algebras such as  group algebra homomorphisms, group algebra representations and Plesken Lie algebra homomorphisms.

\begin{lemma}
	Let $f : G \rightarrow H$ be a group homomorphism.
	Then $\bar{f} : {\mathbb{F}G} \rightarrow {\mathbb{F}H} $ defined by 
	\begin{center}
		$ \bar{f}(\displaystyle\sum_{i}^n a_ig_i)=\displaystyle\sum_{i}^n a_if(g_i)$ 
	\end{center}			
	is a homomorphism between group algebras.
\end{lemma}
\begin{proof}
	Let $G= \{g_1, g_2,. . . , g_n \}$, then ${\mathbb{F}G} = 	\{ \displaystyle\sum_{i=1}^n a_ig_i: a_i \in \mathbb{F} \ \text{for all i} \}$. Clearly $\bar{f}$ is linear and for  $\alpha = \displaystyle\sum_{i=1}^n a_ig_i, \, \beta = \displaystyle\sum_{j=1}^n b_jg_j \in {\mathbb{F}G}$,
	
	\begin{equation*}
		\begin{split}
			\bar{f}(\displaystyle\sum_{i=1}^n a_ig_i\displaystyle\sum_{j} b_jg_j) 
			= \bar{f}(\displaystyle\sum_{i,j=1}^n a_ib_jg_ig_j)
			= \displaystyle\sum_{i,j=1}^n a_ib_jf(g_i)f(g_j) 
			= \displaystyle\sum_{i=1}^n a_if(g_i)\displaystyle\sum_{j=1}^n b_jf(g_j)
		\end{split}		
	\end{equation*}	
	hence, $\bar{f}$ is a homomorphism between group algebras.
\end{proof}	

However, it is not necessary that any group algebra homomorphisms induced by a group homomorphism as seen in the following example.
\begin{exam}
	Consider the group $G = \{e, \sigma \}$ where $\sigma^2 = e$. Define $f : \mathbb{C}G \to \mathbb{C}G$ by
	\begin{center}
		$f(e) = e$ and $f(\sigma) = - \sigma$ with $f$ is linear
	\end{center}
	Then $f$ is a group algebra homomorphism. But $f$ is not induced from any group homomorphism.
\end{exam}

A representation of a group algebra $\mathbb{F}G$ is a linear map $\phi : \mathbb{F}G \to \mathfrak{gl}(V)$ ( where $\mathfrak{gl}(V)$ is the set of all endomorphisms on $V$)  such that $\phi$ is an algebra homomorphism. A representation of a group $G$ gives rise to a representation of the group algebra $\mathbb{F}G$ as follows.
\begin{defn}
	Given a representation $\rho : G \to GL(V)$ of a group $G$. Then $\phi : \mathbb{F}G \to \mathfrak{gl}(V)$ defined by
	\begin{center}
		$\phi(\displaystyle\sum_{i=1}^n a_ig_i) = \displaystyle\sum_{i=1}^n a_i\rho(g_i)$
	\end{center}
	is a representation of $\mathbb{F}G$.
\end{defn} 

A \textit{subrepresentation} of a representation $\phi : \mathbb{F}G \to \mathfrak{gl}(V)$ is a map $\phi' : \mathbb{F}G \to \mathfrak{gl}(W)$ where $W$ is a  vector subspace  $V$ which is invariant under $\mathbb{F}G$. A representation $\phi$ is \textit{irreducible} if there is no proper $\mathbb{F}G$-invariant subspace $W$ of $V$. That is, there is no proper invariant subspace $W$ such that $\phi(x)(w) \in W$ for all $w \in W$ and $x \in \mathbb{F}G$.
\begin{thm}
	If $\rho : G \to GL(V)$ is an irreducible representation of $G$, then the representation  $\phi : \mathbb{F}G \to \mathfrak{gl}(V)$ defined by
	\begin{center}
		$\phi(\displaystyle\sum_{i=1}^n a_ig_i) = \displaystyle\sum_{i=1}^n a_i\rho(g_i)$
	\end{center} is also irreducible.
\end{thm}
\begin{proof}
	Since $\rho : G \to GL(V)$ is an irreducible representation, $V$ has no proper $G$-invariant subspace. Let $\phi : \mathbb{F}G \to \mathfrak{gl}(V)$ be the corresponding representation of $\mathbb{F}G$. Also let $W$ be an invariant subspace of $V$ under the action $\phi$. That is $\phi(x)(w) \in W$ for all $w \in W$ and $x \in \mathbb{F}G$. 
	\begin{equation*}
		\begin{split}
			\phi(x)(w) &= \phi(\displaystyle\sum_{i}^n a_ig_i)(w)\\
			&= \displaystyle\sum_{i}^n a_i\rho(g_i)(w)
		\end{split}
	\end{equation*}
	Inparticular, $\rho(g_i)(w) \in W$ for all $w \in W$ and $g_i \in G$. That is $W$ is an invariant subspace of $V$ under the action $\rho$, which is a contradiction since $\rho$ is irreducible. Hence there is no invariant subspace $W$ of $V$ under the action $\phi$.
\end{proof}
The following example shows that the group algebra representations of the irreducible representations of the symmetric group $S_3$ are also irreducible.
\begin{exam}
	Consider the symmetric group $S_3$ and its irreducible representations $\rho_1, \rho_2$ and $\rho_3$ where $\rho_1$ is the trivial representation with degree 1, $\rho_2$ is the alternating representation with degree 1 and $\rho_3$ is the standard representation with degree 2. Also,
	\begin{equation*}
		\begin{split}
			\rho_1 &: \sigma \mapsto [1] \text{ for all } \sigma \in S_3\\
			\rho_2 &: \text{even permutations maps to [1] and odd permutations maps to [-1]}\\
			\rho_3&: (1\ 2\ 3) \mapsto \begin{bmatrix} \frac{-1}{2}& \frac{\sqrt{3}}{2} \\ \frac{\sqrt{3}}{2}&\frac{-1}{2}\end{bmatrix}, (1\ 2) \mapsto \begin{bmatrix}
				1&0\\0&-1
			\end{bmatrix}
		\end{split}
	\end{equation*}
	Then from \textit{Theorem 2}, the group algebra representations are : \\
	$\phi_1 : \mathbb{R}S_3 \to \mathfrak{gl}(\mathbb{R}) $ given by
	\begin{center}
		$\phi_1(a_1(1) + a_2(1\ 2) + a_3(2\ 3) + a_4(1\ 3) +a_5(1\ 2\ 3) + a_6(1\ 3\ 2)) = (a_1+a_2+a_3+a_4+a_5+a_6)\begin{bmatrix}
			1
		\end{bmatrix}$
	\end{center}
	$\phi_1 : \mathbb{R}S_3 \to \mathfrak{gl}(\mathbb{R}) $ given by
	\begin{center}
		$\phi_2(a_1(1) + a_2(1\ 2) + a_3(2\ 3) + a_4(1\ 3) +a_5(1\ 2\ 3) + a_6(1\ 3\ 2)) = (a_1+a_5+a_6)[1] + (a_2+a_3+a_4)[-1]$
	\end{center}
	$\phi_3 : \mathbb{R}S_3 \to \mathfrak{gl}(\mathbb{R}^2) $ given by
	\begin{center}
		$\phi_2(a_1(1) + a_2(1\ 2) + a_3(2\ 3) + a_4(1\ 3) +a_5(1\ 2\ 3) + a_6(1\ 3\ 2)) = a_1\begin{bmatrix}
			1&0\\0&1
		\end{bmatrix} +a_2 \begin{bmatrix}
			1&0\\0&-1
		\end{bmatrix} +a_3 \begin{bmatrix}
			\frac{-1}{2}&\frac{-\sqrt{3}}{2} \\ \frac{-\sqrt{3}}{2} & \frac{1}{2}
		\end{bmatrix} +a_4 \begin{bmatrix}
			\frac{-1}{2}&\frac{\sqrt{3}}{2} \\ \frac{\sqrt{3}}{2} & \frac{1}{2}
		\end{bmatrix} +a_5 \begin{bmatrix}
			\frac{-1}{2}&\frac{-\sqrt{3}}{2} \\ \frac{\sqrt{3}}{2} & \frac{-1}{2}
		\end{bmatrix} +a_6 \begin{bmatrix}
			\frac{-1}{2}&\frac{\sqrt{3}}{2} \\ \frac{-\sqrt{3}}{2} & \frac{-1}{2}
		\end{bmatrix}$
	\end{center}\vspace{0.1in}
	Since $\phi_1$ amd $\phi_2$ are of degree 1, they are irreducible.  To prove $\phi_3$ is irreducible, let $W$ be a proper  invariant subspace of $\mathbb{R}^2$. Then $W = span\{(\alpha, \beta)\}$ for some $(\alpha, \beta) \in \mathbb{R}^2$. Since $W$ is invariant, $\phi_3(x)(w) \in W$ for all $x \in \mathbb{R}S_3$ and $w \in W$.
	That is, $\phi(\displaystyle \sum_{i=1}^na_i \sigma_i)(w) = k(\alpha, \beta)$ for some $k \in \mathbb{R}$.  From computations we obtained,
	\begin{equation*}
		\begin{split}
			\alpha(a_1+a_2-\frac{a_3}{2} - \frac{a_5}{2} - \frac{a_6}{2}) +\frac{ \sqrt{3}\beta}{2}(a_3 + a_4 - a_5-a_6) &= k \alpha \\
			\frac{\sqrt{3}\alpha}{2}(a_4-a_3+a_5-a_6) + \beta(a_1 - a_2+\frac{a_3}{2} +\frac{a_4}{2} - \frac{a_5}{2} -\frac{a_6}{2}) &= k \beta
		\end{split}
	\end{equation*}
	This is true for every $x \in \mathbb{R}S_3$ and $w \in W$. Thus take $x = \frac{2}{\sqrt{3}} (2\ 3)$, then it can be seen that $\alpha= \beta =0$. Thus $W = 0$ and hence $\phi_3$ is irreducible.
\end{exam}
\begin{defn}	
	Let $G$ be a finite group and $\mathbb{F}$ be a field. The group algebra $\mathbb{F}G$ over $\mathbb{F}$ with the Lie bracket $[\ , \ ] : \mathbb{F}G \times \mathbb{F}G \rightarrow \mathbb{F}G$ defined by  
	\begin{center}
		$[ \alpha  ,  \beta   ] = \alpha\beta  - \beta \alpha$ where $\alpha=\displaystyle\sum_{i} a_ig_i ,   \beta=\displaystyle\sum_{i} b_ig_i $ in $\mathbb{F}G$
	\end{center} 
	is a Lie algebra, denote it by $L_{\mathbb{F}G}$ and is called the \textit{Group Lie algebra}.
\end{defn}	
\par A linear map between two group Lie algebras $L_{\mathbb{F}G}$ and $L_{\mathbb{F}H}$ which preserves the Lie bracket  is a \textit{homomorphism} of group Lie algebras. 

The following lemmas are immediately follow from the definitions.

\begin{lemma}
	Let $f : G \rightarrow H$ be a group homomorphism.
	Then $\bar{f} : L_{\mathbb{F}G} \rightarrow L_{\mathbb{F}H} $ defined by 
	\begin{center}
		$\bar{f}(\displaystyle\sum_{i=1}^n a_ig_i)=\displaystyle\sum_{i=1}^n a_if(g_i)$ 
	\end{center}			
	is a homomorphism between group Lie algebras.
\end{lemma}
\begin{proof}
	Let $G= \{g_1, g_2,. . . , g_n \}$, then $L_{\mathbb{F}G} = 	\{ \displaystyle\sum_{i=1}^n a_ig_i: a_i \in \mathbb{F} \ \text{for all i} \}$. Clearly $\bar{f}$ is linear and for  $\alpha = \displaystyle\sum_{i=1}^n a_ig_i, \, \beta = \displaystyle\sum_{j=1}^n b_jg_j \in L_{\mathbb{F}G}$,
	
	\begin{equation*}
		\begin{split}
			\bar{f}([\alpha, \beta])&=	\bar{f}(\alpha\beta - \beta\alpha)\\
			&= \bar{f}(\displaystyle\sum_{i,j=1}^n a_ib_jg_ig_j - \displaystyle\sum_{j,i=1}^n b_ja_ig_jg_i)\\
			&= \displaystyle\sum_{i,j=1}^n a_ib_jf(g_i)f(g_j) - \displaystyle\sum_{j,i=1}^n b_ja_if(g_j)f(g_i)\\
			&= \displaystyle\sum_{i=1}^n a_if(g_i)\displaystyle\sum_{j=1}^n b_jf(g_j) - \displaystyle\sum_{j=1}^n b_jf(g_j)\displaystyle\sum_{i=1}^n a_if(g_i) \\
			&= [\bar{f}(\alpha), \bar{f}(\beta)]
		\end{split}		
	\end{equation*}	
	hence, $\bar{f}$ is a homomorphism between Lie algebras of group algebras.
\end{proof}	
Clearly the Plesken Lie algebra $\mathcal{L}(G)$ is a proper subset of $L_{\mathbb{F}G}$ and moreover, it is a Lie subalgebra of $L_{\mathbb{F}G}$.

A linear map between two Plesken Lie algebras $\mathcal{L}(G)$ and $\mathcal{L}(H)$ is a \textit{Plesken Lie algebra homomorphism} if it preserves the Lie bracket.
\begin{exam}
	Consider the Plesken Lie algebras $\mathcal{L}(S_3) = \{\alpha((1\ 2\ 3) - (1\ 3\ 2)) : \alpha \in \mathbb{C}\}$  of $S_3$  over $\mathbb{C}$ and $\mathcal{L}(D_4) = \{\alpha(a - a^3) : \alpha \in \mathbb{C}\}$ of  $D_4 = <a, b : a^4 = b^2 =e, aba = b^{-1} >$ over $\mathbb{C}$. Then $\hat{f} : \mathcal{L}(S_3) \to \mathcal{L}(D_4)$ defined by
	\begin{center}
		$\hat{f}(\alpha((1\ 2\ 3) - (1\ 3\ 2))) = \alpha(a - a^3)$
	\end{center}
	is a Plesken Lie algebra homomorphism.
\end{exam}
\begin{lemma}
	Let $f : G \rightarrow H$ be a group homomorphism. Then $\hat{f} : \mathcal{L}(G) \rightarrow \mathcal{L}(H)$ defined by
	\begin{center}
		$\hat{f}(\displaystyle \sum_{i=1}^n a_i \hat{g_i}) = \displaystyle \sum_{i=1}^n a_i \widehat{f(g_i)}$
	\end{center}
	is a Plesken Lie algebra homomorphism. Further, $\hat{f}$ is actually the restriction of the homomorphism of the group Lie algebras $\bar{f} : L_{\mathbb{F}G} \to L_{\mathbb{F}H}$(defined in \textit{Lemma 3}) to $\mathcal{L}(G)$. That is, 
	\begin{center}
		$\hat{f} = \bar{f} \arrowvert_{\mathcal{L}(G)}$
	\end{center}
\end{lemma}
\begin{proof}
	The proof is similar as in \textit{Lemma 3}.
\end{proof}

\section{Plesken Lie algebra representations}
Next we proceed to define representation of a Plesken Lie algebra and discuss some of its properties. We have already described the representations of group algebras and it is seen that when the representations of group is irreducible so is the representation of group algebras.
\begin{defn}
	A representation of a Plesken Lie algeba $\mathcal{L}(G)$ is a linear map $\phi : \mathcal{L}(G) \to \mathfrak{gl}(V)$ (where $V$ is a vector space over $\mathbb{F}$) such that $\phi$ is a Lie algebra homomorphism.
\end{defn}
\par The following theorem shows that if we have a representation of a group $G$, then we can find a representation of the Plesken Lie algebra $\mathcal{L}(G)$.
\begin{lemma}
	
	If $\rho : G \to GL(V)$ is a representation of $G$ on $V$, then $\psi : \mathcal{L}(G) \to \mathfrak{gl}(V)$ defined by
	\begin{center}
		
		$\psi(\displaystyle \sum_{i=1}^n a_i \hat{g_i}) = \displaystyle \sum_{i=1}^n a_i \widehat{\rho(g_i)}$  for $\displaystyle \sum_{i=1}^n a_i \hat{g_i} \in \mathcal{L}(G)$
	\end{center}
	is a representation of the Plesken Lie algebra $\mathcal{L}(G)$.
\end{lemma}
\begin{proof}
	Since $\rho : G \to GL(V)$ is a repesentation of $G$, $\rho(g_i)$ and $\rho({g_i}^{-1})$ are automorphisms on $V$. Then 
	\begin{center}
		
		$\widehat{\rho(g_i)} =( \rho(g_i) - {\rho(g_i)}^{-1}) = \rho(g_i) - \rho({g_i}^{-1})$.
	\end{center}
	is an endomorphism on $V$. Thus $\psi$ is defined.\\
	Clearly $\psi$ is linear. By $Lemma \ 4$, $\psi$ is a homomorphism from  $\mathcal{L}(G)$ to $\mathfrak{gl}(V)$ and hence $\psi$ is a representation of $\mathcal{L}(G)$.
	
\end{proof}
\par A Plesken Lie algebra representation $\psi : \mathcal{L}(G) \to \mathfrak{gl}(V)$ is \textit{irreducible} if there is no proper $\mathcal{L}(G)$-invariant subspace $W$ of $V$.

The following theorem states that the Plesken Lie algebra representation corresponding to a reducible group representation is reducible.
\begin{thm}
	If $\rho : G \to GL(V)$ is a reducible representation of a group $G$, then the Plesken Lie algebra representation $\psi: \mathcal{L}(G) \to \mathfrak{gl}(V)$ of $\mathcal{L}(G)$ defined by
	\begin{center}
		$\psi(\displaystyle \sum_{i=1}^n a_i \hat{g_i}) = \displaystyle \sum_{i=1}^n a_i \widehat{\rho(g_i)}$
	\end{center} is also reducible.
\end{thm}
\begin{proof}
	Since $\rho : G \to GL(V)$ is a reducible representation, $V$ has a proper invariant subspace.  That is, there is a subspace $W$ of $V$ such that $\rho(g)w \in W$ for all $g \in G$ and $w \in W$. We claim that $W$ is also invariant under $\psi$.  For let  $ \hat{x}= \displaystyle \sum_{i=1}^na_i\hat{g_i} \in \mathcal{L}(G), w \in W $,
	\begin{equation*}
		\begin{split}
			\psi(\hat{x})(w) =  \psi(\displaystyle \sum_{i=1}^na_i\hat{g_i})(w)
			= \displaystyle \sum_{i=1}^na_i\widehat{\rho(g_i)}(w)
			= \displaystyle \sum_{i=1}^na_i(\widehat{\rho(g_i)}(w) -\widehat{\rho({g_i}^{-1})}(w) )
		\end{split}
	\end{equation*}
	Since $\rho(g_i)(w), \rho(g_i^{-1})(w) \in W$ and $W$ is a subspace, $\widehat{\rho(g_i)}(w) -\widehat{\rho({g_i}^{-1})}(w)\in W$. Thus $\psi(\hat{x})w \in W$ for all $\hat{x} \in \mathcal{L}(G)$ and $w \in W$.
\end{proof}

\par Very often if $\rho$ is an irreducible representation of a group $G$, then $\psi$ is also an irreducible representation as seen in the following examples(where $\psi$ is defined as in $Lemma \ 5$).

\begin{exam}
	Consider the irreducible representations of $S_3$ as in \textit{Example 3}.
	Then corresponding Plesken Lie algebra representations are:
	\begin{equation*}
		\begin{split}
			\psi_1 : \mathcal{L}(S_3) \to \mathfrak{gl}(\mathbb{R}) \text{ given by } \psi_1(a((1\ 2\ 3) - (1\ 3\ 2))) &= 0\\
			\psi_2 : \mathcal{L}(S_3) \to \mathfrak{gl}(\mathbb{R}) \text{ given by } \psi_2(a((1\ 2\ 3) - (1\ 3\ 2))) &= 0 \\
			\psi_3 : \mathcal{L}(S_3) \to \mathfrak{gl}(\mathbb{R}^2) \text{ given by } \psi_3(a((1\ 2\ 3) - (1\ 3\ 2))) &= a(\begin{pmatrix} \frac{-1}{2}& \frac{\sqrt{3}}{2} \\ \frac{\sqrt{3}}{2}&\frac{-1}{2}\end{pmatrix} - \begin{pmatrix} \frac{-1}{2}& \frac{\sqrt{3}}{2} \\ \frac{\sqrt{3}}{2}&\frac{-1}{2}\end{pmatrix}\\
			&=a(\begin{pmatrix}0&-\sqrt{3}\\ \sqrt{3}&0 \end{pmatrix})
		\end{split}
	\end{equation*}
	Clearly, $\psi_1$ and $\psi_2$ are irreducible representations since each of which has degree 1. We will prove $\psi_3$ is irreducible. For $W$ be a proper  invariant subspace of $\mathbb{R}^2$. Then $W = span\{(\alpha, \beta)\}$ for some $(\alpha, \beta) \in \mathbb{R}^2$. Since $W$ is invariant, $\psi_3(\hat{x})(w) \in W$ for all $\hat{x} \in \mathcal{L}(S_3)$ and $w \in W$.
	\begin{equation*}
		\begin{split}
			\psi_3(\hat{x})(w)\in W &\Rightarrow a(\begin{pmatrix}0&-\sqrt{3}\\ \sqrt{3}&0 \end{pmatrix})\begin{pmatrix}\alpha\\ \beta \end{pmatrix} \in W\\
			& \Rightarrow a(\begin{pmatrix}
				-\sqrt{3}\beta \\ \sqrt{3} \alpha
			\end{pmatrix}) \in W \\
			& \Rightarrow (-\sqrt{3}a\beta , \sqrt{3} a\alpha) = k(\alpha, \beta) \text{ for some $k \in \mathbb{R}$ }\\
			& \Rightarrow k = \frac{-\sqrt{3}a \beta}{\alpha} \text{ and } k = \frac{\sqrt{3}a \alpha}{\beta}
		\end{split}
	\end{equation*}
	By equating we get, $\sqrt{3}a(\alpha^2 + \beta^2) = 0$ which implies $\alpha = \beta =0$. That is, $W = 0$, hence $\psi_3$ is an irreducible representation.
\end{exam}
\begin{exam}
	Consider the dihedral group $D_4 = <a, b : a^4=b^2=e, aba = b >$ and its irreducible representations $\rho_1, \rho_2, \rho_3, \rho_4$ and $\rho_5$ where $\rho_i$'s are given by
	\begin{equation*}
		\begin{split}
			\rho_1 &: a \mapsto (1), b \mapsto (1) \\
			\rho_2 &: a \mapsto (-1), b \mapsto (1) \\
			\rho_3 &: a\mapsto (1), b \mapsto (-1) \\
			\rho_4 &: a \mapsto (-1), b \mapsto (-1) \\
			\rho_5 &: a\mapsto \begin{pmatrix}
				0&-1\\1&0
			\end{pmatrix},
			b \mapsto \begin{pmatrix}
				0&1\\1&0
			\end{pmatrix}
		\end{split}
	\end{equation*}
	Then corresponding Plesken Lie algebra representations are:
	\begin{equation*}
		\begin{split}
			\psi_1 &: \mathcal{L}(D_4) \to \mathfrak{gl}(\mathbb{R}) \text{ \ given by \ } \psi(\alpha(a - a^3)) = 0 \\
			\psi_2 &: \mathcal{L}(D_4) \to \mathfrak{gl}(\mathbb{R}) \text{ \ given by \ } \psi(\alpha(a - a^3)) = 0 \\
			\psi_3 &: \mathcal{L}(D_4) \to \mathfrak{gl}(\mathbb{R}) \text{ \ given by \ } \psi(\alpha(a - a^3)) = 0 \\
			\psi_4 &: \mathcal{L}(D_4) \to \mathfrak{gl}(\mathbb{R}) \text{ \ given by \ } \psi(\alpha(a - a^3)) = 0 \\
			\psi_5 &: \mathcal{L}(D_4) \to \mathfrak{gl}(\mathbb{R^2}) \text{ \ given by \ } \psi(\alpha(a - a^3)) = \alpha(\begin{pmatrix}
				0&-2\\2&0
			\end{pmatrix})
		\end{split}
	\end{equation*}
	Since $\psi_1,\psi_2, \psi_3$ and $\psi_4$ has degree 1, they are irreducible. Next we will prove $\psi_5$ is also an irreducible representation. For let $W$ be a proper invariant subspace of $\mathbb{R}^2$. Then $W = span{(\alpha, \beta)}$ for some  $(\alpha, \beta) \in \mathbb{R}^2$. Since $W$ is invariant, $\psi_5(\hat{x})(w) \in W$ for all $\hat{x} \in \mathcal{L}(D_4)$ and $w \in W$.
	\begin{equation*}
		\begin{split}
			\psi_5(\hat{x})(w)\in W &\Rightarrow \gamma (\begin{pmatrix}0&-2\\ 2&0 \end{pmatrix})\begin{pmatrix}\alpha\\ \beta \end{pmatrix} \in W\\
			& \Rightarrow \gamma (\begin{pmatrix}
				-2\beta \\ 2 \alpha 
			\end{pmatrix}) \in W \\
			& \Rightarrow (-2\gamma\beta , 2 \gamma\alpha) = k(\alpha, \beta) \text{ for some $k \in \mathbb{R}$ }\\
			& \Rightarrow k = \frac{-2\gamma \beta}{\alpha} \text{ and } k = \frac{2\gamma  \alpha}{\beta}
		\end{split}
	\end{equation*}
	By equating we get, $2\gamma(\alpha^2 + \beta^2) = 0$ which implies $\alpha = \beta =0$. That is, $W = 0$, hence $\psi_5$ is an irreducible representation.
\end{exam}
Also we obtained that the representations of $\mathcal{L}(S_4)$ corresponding to the irreducible representations of $S_4$ are irreducible.

However, the above situation is not a necessary and sufficient condition, because it is possible to have $\psi$ may be $\mathcal{L}(G)$-reducible even if $\rho$ is $\mathbb{F}G$-irreducible.

\subsection{Plesken Lie algebra modules}
Next we proceed to describes Plesken Lie algebra modules and obtain some interesting theories such as Schur's lemma.
\begin{defn}
	A vector space $V$, endowed with an operation $\mathcal{L}(G) \times V \to V$ is an $\mathcal{L}(G)$-module if
	\begin{enumerate}
		\item $(a \hat{x} + b \hat{y})v = a(\hat{x}v) + b(\hat{y}v)$
		\item $\hat{x}(av + bw) = a(\hat{x}v) + b(\hat{x}v)$
		\item $[\hat{x}, \hat{y}] = \hat{x} \hat{y} v - \hat{y}\hat{x} v$
	\end{enumerate}
	for all $\hat{x}, \hat{y} \in \mathcal{L}(G), v, w \in V$ and $ a, b \in \mathbb{F}$.
\end{defn}
\begin{rmk}
	Every $\mathbb{F}G$ is an $\mathcal{L}(G)$-module.
\end{rmk}
\begin{proof}
	Suppose $V$ is an $\mathbb{F}G$-module. Then for any $\hat{x} \in \mathcal{L}(G)$ and $v \in V$,
	\begin{center}
		$\hat{x}v = (\displaystyle\sum_{i=1}^n\hat{g_i})v = \displaystyle\sum_{i=1}^na_i(\hat{g_i}v)$
	\end{center}
	Since $V$ is an $\mathbb{F}G$-module, $gv \in V$ for all $g \in G$ and $v \in V$. Thus $\hat{g_i}v = g_iv - {g_i}^{-1}v \in V$ which implies $\hat{x}v \in V$ for all $\hat{x} \in \mathcal{L}(G)$ and $v \in V$ and this satisfies all the axioms of an $\mathcal{L}(G)$-module. Thus $V$ is an $\mathcal{L}(G)$-module.
\end{proof}

Note that the converse of the remark need not be true.
\begin{exam}
	Consider the group $G = <a : a^3=e> = \{e, a, a^2\}$ and the vector space $V = span\{v_1, v_2, v_3\}$ over the field $\mathbb{F}$. define the product $g_iv$ as follows : 
	\begin{center}
		$ev_1 = v_1 ,\quad  ev_2 = v_2,\quad ev_3 = v_3$ \\
		$av_1 = v_2,\quad av_2 = v_3 ,\quad av_3 = v_1$ \\
		$a^2v_1 = v_3,\quad a^2v_2 = -v_1,\quad a^2v_3 = v_2$
	\end{center}
\end{exam}
Then $V$ is an $\mathbb{F}G$-module and thus it is an $\mathcal{L}(G)$-module. Define
\begin{center}
	$U = span\{v_1 + v_2\}$
\end{center}
Then $U$ is a subspace of $V$. Since $a(v_1 + v_2) = v_2 + v_3 \notin U$, $U$ is not an $\mathbb{F}G$-submodule. \\
We have $\mathcal{L}(G) = \{ \alpha (a - a^2) : \alpha \in \mathbb{F} \}$. Then for any $\hat{x} = \alpha(a-a^2) \in \mathcal{L}(G)$, and $u=v_1 +v_2 \in U$, 
\begin{equation*}
	\begin{split}
		\alpha (a -a^2)(v_1+v_2) &= \alpha a(v_1+v_2) - \alpha a^2(v_1+v_2) \\
		&= \alpha (v_2+v_3-(v_3-v_1)) \\
		& = \alpha (v_2 + v_1)  \in U
	\end{split}         
\end{equation*}
Thus $U$ is an $\mathcal{L}(G)$-submodule of $V$.

\begin{thm}[Schur's lemma]
	Let $V$ and $W$ are irreducible $\mathcal{L}(G)$-modules.
	\begin{enumerate}
		\item[(1)] If $\theta : V \to W$ is an $\mathcal{L}(G)$-homomorphism, then either $\theta$ is an $\mathcal{L}(G)$-isomorphism or $\theta(v) = 0$ for all $v \in V$.
		\item[(2)] If $\theta : V \to V$ is an $\mathcal{L}(G)$-isomorphism, then $\theta(v) = \lambda_vI_V$.
	\end{enumerate}
\end{thm}
\begin{proof}
	\begin{itemize}
		\item[(1)] Suppose $\theta(v) = 0$ for some $v \in V$. Then $Im(\theta) \neq \{0\}$. Since $Im(\theta)$ is a submodule of $W$  and $W$ is irreducible, $Im(\theta) = W$. Thus $\theta$ is onto. Since $Ker(\theta)$ is a submodule of $V$, $V$ is irreducible and $Ker(\theta) \subsetneq V$, $Ker(\theta) = \{0\}$. Thus $\theta$ is one-one. Hence $\theta$ is an isomorphism. 
		\item[(2)] Suppose $\theta : V \to V$ is an $\mathcal{L}(G)$-isomorphism. Then $\theta$ has an eigen value $\lambda \in \mathbb{C}$. Thus $Ker(\theta - \lambda I_V)\neq \{0\}$. $ Ker(\theta - \lambda I_V)$ is a non-zero $\mathcal{L}(G)$-submodule of $V$ and since $V$ is irreducible, $Ker(\theta - \lambda_v I_V) = V$. Hence, $\theta = \lambda I_V$.
	\end{itemize}

\end{proof}
\begin{thm}
	If $G$ is an abelian group, then every irreducible $\mathcal{L}(G)$-module has dimension 1.
\end{thm}
\begin{proof}
	Suppose $G$ is an abelian group. Then $\mathcal{L}(G)$ is an abelian Plesken Lie algebra. That is, $[\hat{x}, \hat{y}] = 0$ for all $\hat{x}, \hat{y} \in \mathcal{L}(G)$ which implies that $\hat{x}\hat{y} = \hat{y}\hat{x}$ for all $ \hat{x}, \hat{y} \in \mathcal{L}(G)$. Let $V$ be an irreducible $\mathcal{L}(G)$-module. Define $\phi : V \to V$ by 
	\begin{center}
		$\phi(v)  = v \hat{x}$
	\end{center}
	Then $\phi$ is linear and 
	\begin{center}
		$\phi(v\hat{y}) = (v\hat{y})\hat{x} = v(\hat{y}\hat{x}) = v(\hat{x}\hat{y}) = (v \hat{x})\hat{y}= \phi(v)\hat{y}$
	\end{center}
	Thus $\phi$ is an $\mathcal{L}(G)$-module homomorphism on $V$. Then by Schur's lemma, $\phi = \lambda I_V$ for some $\lambda \in \mathbb{C}$. That is, $v\hat{x} = \lambda v$ for all $ v \in V$. That is, every subspace of $V$ is an $\mathcal{L}(G)$-submodule of $V$. Since $V$ is irreducible, $dim(V) = 1$.	
\end{proof}

An $\mathcal{L}(G)$-module gives us many representations of Plesken Lie algebras, all of the form 
\begin{center}
	$\hat{x} \mapsto [\hat{x}]_{\mathfrak{B}} \quad ( \hat{x} \in \mathcal{L}(G) )$
\end{center}
for some $\mathfrak{B}$ of $V$. Also if $\phi : \mathcal{L}(G) \to \mathfrak{gl}(V)$ is a representation of $\mathcal{L}(G)$, then $V$ can be viewed as an $\mathcal{L}(G)$-module via the action $\hat{x}v = \phi(\hat{x})(v)$. That is, there is a one-to-one correspondence between $\mathcal{L}(G)$-modules and representations of the Plesken Lie algebra $\mathcal{L}(G)$.

\begin{thm}
	Let $\phi : \mathcal{L}(G) \to \mathfrak{gl}(V)$ be an irreducible representation. Then the only endomorphisms of $V$ commuting with all $\phi(\hat{x})$($\hat{x} \in \mathcal{L}(G)$) are the scalars.
\end{thm}
\begin{proof}
	Suppose $y \in \mathfrak{gl}(V)$ commutes with matrices in $\phi(\mathcal{L}(G))$.
	Let $\lambda$ be the eigen value of $y$ and $V_{\lambda}$, the eigen space of $y$ relative to $\lambda$. For any $\hat{x} \in \mathcal{L}(G)$,
	\begin{center}
		$y(\phi(\hat{x})v) = \phi(\hat{x})yv = \phi(\hat{x})\lambda v = \lambda (\phi(\hat{x})v)$
	\end{center}
	That is, $\phi(\hat{x})v \in V_{\lambda}$. We have $V_{\lambda} = Ker(y - \lambda I_V)$ is an $\mathcal{L}(G)$-submodule of $V$. Since $V$ is irreducible and $0 \neq \phi(\hat{x})v \in V_{\lambda}$, $V_{\lambda} = V$, which implies $y = \lambda$, a scalar matrix.
\end{proof}

\end{document}